\documentclass{article}
\usepackage[a4paper, total={6in, 9in}]{geometry}
\usepackage{graphicx} 
\usepackage{amsmath,amscd,mathtools,cases,amsthm,amsopn,amssymb,soul,cleveref,mathabx,relsize,stmaryrd}
\usepackage{tikz}
\usepackage{xcolor,chngcntr}
\counterwithin{figure}{section}
\usetikzlibrary{calc}
\tikzset{bzf/.style={circle, draw, fill=blue!30,inner sep=2pt}, minimum size=2pt,inner sep=2pt}
\tikzset{rzf/.style={circle, draw, fill=red!30,inner sep=2pt}, minimum size=2pt,inner sep=2pt}
\tikzset{wzf/.style={circle, draw,inner sep=2pt}}
\tikzstyle{vertex}=[circle, draw, inner sep=0pt, minimum size=4pt,fill=black]

\tikzstyle{whitevertex}=[circle, draw, inner sep=0pt, minimum size=4pt,fill=white]
\tikzstyle{bluevertex}=[circle, draw, inner sep=0pt, minimum size=4pt,fill=black]

\tikzstyle{hollowvertex}=[circle, draw, inner sep=1pt, minimum size=4pt, fill=white]

\tikzstyle{phantomvertex}=[circle, draw, inner sep=0pt, minimum size=4pt,color=white]

\usetikzlibrary{backgrounds}
\usetikzlibrary{arrows.meta}
\tikzset{
  .../.tip={[sep=2pt 2]
    Round Cap[]. Circle[length=0pt 2,sep=2pt] Circle[length=0pt 2,sep=2pt] Circle[length=0pt 2, sep=2pt 2]}}
\tikzset{
  .../.tip={[sep=2pt 2]
    Round Cap[]. Circle[length=0pt 2,sep=2pt] Circle[length=0pt 2,sep=2pt] Circle[length=0pt 2, sep=2pt 2]}}
\tikzset{
      ellipsis/.tip={
Square[length=2pt,sep=0pt,color=white] Circle[length=1pt,sep=0pt,color=black] Square[length=1pt,sep=0pt,color=white]
Circle[length=1pt,sep=0pt,color=black] Square[length=1pt,sep=0pt,color=white]
Circle[length=1pt,sep=0pt,color=black] Square[length=2pt,sep=0pt,color=white]}}
\tikzset{middlearrow/.style n args={2}{
        decoration={markings,
            mark= at position {#2} with {\arrow{#1}} ,
        },
        postaction={decorate}
    }
}
\usetikzlibrary{positioning}
\usepackage{url}

\newtheorem{theorem}[subsection]{Theorem}
\newtheorem{corollary}[subsection]{Corollary}
\newtheorem{observation}[subsection]{Observation}
\newtheorem{proposition}[subsection]{Proposition}
\newtheorem{lemma}[subsection]{Lemma}

\theoremstyle{definition}
\newtheorem{definition}[subsection]{Definition}

\newtheorem{example}[subsection]{Example}

\DeclareMathOperator{\Z}{Z}

\DeclareMathOperator{\Es}{Es}
\usepackage{authblk}

\usepackage{lineno}

\title{Vertex Fault Tolerant Zero Forcing}
\author[1]{Asher Brown}
\author[1]{Mark Hunnell}
\author[1]{Za'Kiya Toomer-Sanders}
\author[1]{Via Weber}
\affil[1]{Winston-Salem State University}
\begin{document}
\maketitle

\begin{abstract}
   Zero forcing is an iterative graph coloring process studied for its wide array of applications. In this process, the vertices of the graph are initially designated as blue or white, and a zero forcing set is a set of initially blue vertices that results in all vertices becoming blue after repeated application of a color change rule. The zero forcing number of a graph is the minimum cardinality of a zero forcing set. The zero forcing number has motivated the introduction of a host of variants motivated by linear-algebraic or graph-theoretic contexts. We define a variant we term the $k$-fault tolerant zero forcing number, which is the minimum cardinality of a set $B$ such that every subset of $B$ of cardinality $|B|-k$ is a zero forcing set.  We study the values of this parameter on various graph families, the behavior under several graph operations, and characterize the 1-fault tolerant zero forcing number of trees.
\end{abstract}


\section{Introduction}

Zero forcing on a graph was introduced in \cite{AIM2008} to bound the maximum nullity obtained in space of matrices defined by the graph.  We delay formal definitions until \Cref{sec:prelim}, but informally the goal of zero forcing is to determine the smallest size of a set of vertices that when these vertices are the only vertices initially colored blue, the remaining vertices all become blue after a finite number of applications of the zero forcing color change rule.

Since its introduction, there has been significant interest in zero forcing and a host of related parameters that are defined using a similar paradigm.  Though first studied for its applications to the problem of determining the minimum rank of a symmetric matrix associated with a graph \cite{AIM2008, allison2010minimum, parameters2013,ekstrand2013psdzf}, zero forcing and its variants have also been studied for their applications to quantum control \cite{quantum}, the spread of disease and information \cite{DREYER2009}, and monitoring electric grids\cite{haynes2002domination}. The latter application is related to the placement of phasor measurement units (PMUs) in an electrical power network, where Kirchhoff's current and voltage laws allow one to deduce the state of the system at each node from the measured state at the nodes with PMUs alone. This inference process corresponds precisely to the zero forcing rule with blue vertices corresponding to nodes with PMUs. In a real-world implementation, however, PMUs are taken offline periodically for maintenance or could be subject to power failure or communication loss. Placing the minimum number of PMUs risks losing observability of the network under these conditions, and a more robust framework is desirable. This application has motivated investigation into variations of the parameter when PMUs are allowed to fail \cite{bjorkman2023robust,pai2010restricted}. This perspective motivates our definition of fault tolerant zero forcing below. 

Robustness is a desirable property when searching a graph, and in this article we introduce a fault tolerant variant of zero forcing and study its properties. Vertices in a zero forcing set may correspond to the placement of sensors to monitor a network similar to the PMU placement problem, and in this analogy it is natural to want to maintain knowledge of the network status when sensors fail unexpectedly or are taken offline for routine maintenance. While the leaky-forcing variant of zero forcing studied in \cite{abbas2025,alameda2024leaky,alameda2023gen} seeks to provide robustness of the zero forcing process when certain vertices are not allowed to contribute to the zero forcing process, our parameter is determined by the behavior of all subsets of a certain size within a prescribed set of vertices.  Thus our parameter is novel and distinct.

In \cite{KENTER2019124}, zero forcing is related to the solving of linear systems, and leaky forcing was defined as a parameter to study the resilience of solving linear or physical systems where a leak at a node prevents the information available at that node from being able to be used reliably to deduce the state of other vertices in the system. In contrast, the vertex fault tolerant zero forcing parameter studied herein corresponds to a loss of information at a node (or nodes) but not the reliability of information about the node through the rest of the system. We return to the PMU setting to make the distinction concrete: a leak at a vertex corresponds to a sensor whose own reading cannot be trusted for inference even while it remains in place, whereas the failures we consider correspond to a sensor that has simply gone offline, taking its reading with it but casting no doubt on the readings collected elsewhere. It is this latter failure mode — planned or unplanned outages of otherwise-reliable monitoring units — that fault tolerant zero forcing is designed to certify against.

Much is known about the zero forcing number, including the values of the parameter for many families of graphs \cite{AIM2008}, as well as its behavior under various graphs operations, see \cite{hogben2022book} for a detailed account.  The zero forcing polynomial, whose coefficient of $x^k$ is determined by the number of zero forcing sets of size $k$, was studied in \cite{BOYER201935}. In \cite{MENON2025114516}, the behavior of the zero forcing polynomial under a number of graph operations is investigated. Of particular relevance to our variant is the distinction between minimum and minimal zero forcing sets, which was explored in \cite{minimal}. Additionally, other parameters of a graph derived directly from the zero forcing number and the zero forcing color change rule have been investigated.  Examples include the propagation time \cite{hogben2012prop} of a graph and the throttling number \cite{carlson2019throttling} of a graph. We thus have many tools available for studying the fault tolerant zero forcing parameter.  

The remainder of the article is organized as follows. We give the foundational definitions in \Cref{sec:prelim}, as well as make a number of observations useful to the study of fault tolerant zero forcing.  \Cref{sec:families} is devoted to determining the fault tolerant zero forcing number for several common graph families. A general lower bound on the fault tolerant zero forcing number is given in \Cref{sec:pathcovers}, where we also characterize the fault tolerant zero forcing number for trees. Finally, the effect of various graph operations on the fault tolerant zero forcing number is explored in \Cref{sec:ops}. 

\section{Preliminaries}\label{sec:prelim}

Let $G$ be a simple graph (i.e. undirected with no loops or multiple edges) and denote its set of vertices by $V(G)$ and its edges by $E(G)$.  We denote by $uv$ the edge of $G$ whose endpoints are the vertices $u,v \in V(G)$, while $\{u,v\}$ denotes the subset of $V(G)$ with these vertices.  For $V\subseteq V(G)$, we denote the subgraph of $G$ induced by $V$ by $G[V]$.  We follow the notation in \cite{hogben2022book}, and in particular for a graph with $n$ vertices, we denote path graphs by $P_n$, cycle graphs by $C_n$, and complete graphs by $K_n$. The complete bipartite graph with partite sets of $m$ and $n$ is denoted $K_{m,n}$. We denote the degree of a vertex $v \in V(G)$ by $\deg_G(v)$ or $\deg(v)$ if the context is clear. A vertex $v$ is called universal if $\deg(v)=n-1$, and $v$ is called a leaf if $\deg(v)=1$.  Wheel graphs $W_n$ are constructed by adding a universal vertex to $C_{n-1}$. The symbols $N(v)$ and $N[v]$ denote the open and closed neighborhoods of $v$, respectively. Two vertices $u$ and $v$ of $G$ are twins if $N_G(u)\setminus\{v\}=N_G(v)\setminus\{u\}$. 

We choose an initial set of vertices $B\subseteq V(G)$ to designate as blue, and the complement of $B$ in $V(G)$, denoted $W$, are designated white.  The zero forcing color change rule allows a blue vertex $b \in B$ to change the color of a white vertex $w \in W$ to blue if $w$ is the unique white neighbor of $b$.  In this case, we say that $b$ forces $w$ and this is denoted by $b \rightarrow w$.  The zero forcing color change rule is applied iteratively until no more forces are possible.  We say that the initial set of blue vertices $B$ is a zero forcing set for $G$ if eventually all of $V(G)$ is colored blue under repeated application of the color change rule, and $B$ is a failed zero forcing set otherwise. Note that the order in which valid forces are performed does not effect whether a set of vertices is a zero forcing set or failed zero forcing set. An example is shown in \Cref{zfs}.

\begin{figure}[ht]
\centering
\begin{tikzpicture}[scale=.75]
		\node [bzf] (1) at (0, 0) { $1$};

		\node [bzf] (2) at (1, -1) { $2$};

		\node [wzf] (3) at (1, 1) { $3$};

		\node [wzf] (4) at (2, 0) {$4$ };

		\node [wzf] (5) at (3.5, 0) {$5$ };
		\node [ ] (a) at (1,-1.75) {(a)};

\draw[thick] (1) -- (2) -- (4) -- (3) -- (1);
\draw[thick] (4)--(5);

\begin{scope}[shift= { ( 6,0)}]
		\node [wzf] (1) at (0, 0) {$1$ };

		\node [bzf] (2) at (1, -1) {$2$ };

		\node [bzf] (3) at (1, 1) {$3$ };

		\node [wzf] (4) at (2, 0) {$4$ };

		\node [wzf] (5) at (3.5,0) {$5$ };
				\node  (a) at (1,-1.75) {$(b)$};

\draw[thick] (1) -- (2) -- (4) -- (3) -- (1);
\draw[thick] (4)--(5);

\end{scope}
\end{tikzpicture}
\caption{A minimum zero forcing set for $G$ in $(a)$, and a failed zero forcing set for $G$ in $(b)$.}
\label{zfs}
\end{figure}

Clearly, one could choose $B=V(G)$ to obtain a zero forcing set, so the interesting cases are when $B$ contains as few vertices as possible.  This leads to the definition of $\Z(G)$, the zero forcing number of a graph $G$.

\begin{definition}
    Let $G$ be a graph.  The zero forcing number of $G$, denoted $\Z(G)$, is the minimum cardinality of a zero forcing set for $G$.
\end{definition}

A zero forcing set can be thought of as a set of vertices such that if that status of the vertices in the set are known, then the status of every vertex in the graph is determined. For instance, in a space of matrices described by the graph, if the values of entries in a null vector corresponding to a zero forcing set are chosen to be zero, the null vector must be the zero vector.  See \cite{hogben2022book} for full details.

By introducing a robustness requirement, as in our next definition, we ensure that a change in status of one or more of the vertices in a zero forcing set does not change the deterministic nature of the set.

\begin{definition}
Let $G$ be a graph and $B$ a subset of the vertices $V(G)$.  Suppose $B$ has $m$ elements, i.e. $| B | = m$.  Then  $B$ is $k$-fault tolerant if every subset $F \subseteq B$ such that $| F | = m-k$ is also a zero forcing set.
\end{definition}

\begin{definition} Let $G$ be a graph.  The $k$-fault tolerant zero forcing number of $G$ is the minimum cardinality of a $k$-fault tolerant zero forcing set for $G$.  i.e.
\[ \Z_{t}^k(G)  = \min\left\{ | B | \text{ such that } B \text{ is $k$-fault tolerant}  \right\}        \]
\end{definition}

{Thus when $k=0$, we recover the usual zero forcing number of a graph $G$, i.e. $\Z_t^0(G) = \Z(G)$.} When $k=1$, we will use $\Z_t(G)$ to denote $\Z_t^1(G)$. In the remainder of this section, we will make some simple observations about the fault tolerant zero forcing number.  Though each is quite straight forward, they are nevertheless useful in later sections.

First, we observe that the fault tolerant zero forcing number operates independently over connected components of graphs.
\begin{observation}
    If $G$ is a graph, its connected components are $G_1, G_2, \dots, G_r$, and $\Z_t^k(G_i)$ exists for $i=1,2,\dots,r$, then \[\Z_t^k(G) = \displaystyle \sum_{i=1}^r \Z_t^k(G_i).\]
\end{observation}

In light of this observation, there is no loss of generality in considering only connected graphs. It is natural to look for bounds on the fault tolerant zero forcing number $\Z_t^k(G)$ in terms of the zero forcing number $\Z(G)$, and our next observation establishes such a bound.

\begin{lemma}\label{lowbound}
Let $G$ be a graph and suppose $\Z_t^k(G)$ exists.  Then $\Z_t^k(G) \geq \Z(G) + k$. Moreover, \[\Z_t^k(G) \geq \Z_t^{k-1}(G) \geq \dots \geq \Z_t(G) > \Z(G).\]
\end{lemma}
 \begin{proof}
    For the first statement, since $\Z(G)$ is defined to be the minimum cardinality of a zero forcing set, any $k$-fault tolerant set of cardinality less than $\Z(G) + k$ would contradict the minimality of $\Z(G)$. Since any superset of a zero forcing set is itself a zero forcing set, observe that any $k$-fault tolerant zero forcing set is also $(k-1)$-fault tolerant. Thus the second statement follows.
 \end{proof}

We will show in \Cref{sec:families} that several common graph families, such as paths, stars, and complete graphs, have equality with this lower bound when $k=1$. Note that these families all have extreme zero forcing numbers relative to their order.  When the zero forcing number of a graph is lower, such as cycles and wheels, inequality is obtained.  See \Cref{sec:families} for more details.

If a graph $G$ contains two minimum zero forcing sets with no vertices in common, then choosing the union of those sets provides a fault tolerant zero forcing set and establishes that $\Z_t(G) \leq 2\Z(G)$, but in practice this is not a very useful bound.

\begin{definition}
    Let $G$ be a graph and $B$ a zero forcing set of $G$.  We say $B$ is a minimal zero forcing set if for all $v\in B$, $B \setminus \{v\}$ is not a zero forcing set of $G$.
\end{definition}

An important distinction exists between minimum zero forcing sets and minimal zero forcing sets.  To illustrate the difference, we consider the path on $7$ vertices, labeled as in \Cref{P7}.  There are exactly two minimum zero forcing sets, namely $\{v_1\}$ and $\{v_7\}$, and these sets are also minimal.  There are four additional minimal zero forcing sets of $G$, namely $\{v_2,v_3\}$, $\{v_3,v_4\}$, $\{v_4,v_5\}$, and $\{v_5,v_6\}$.

\begin{figure}
    \centering
    \begin{tikzpicture}
        \node[wzf] (a) at (0,0) {$v_1$};
        \node[wzf] (b) at (1.5,0) {$v_2$};
        \node[wzf] (c) at (3,0) {$v_3$};
        \node[wzf] (d) at (4.5,0) {$v_4$};
        \node[wzf] (e) at (6,0) {$v_5$};
        \node[wzf] (f) at  (7.5,0) {$v_6$};
        \node[wzf] (g) at (9,0) {$v_7$};
        \draw (a) -- (b) -- (c) -- (d) -- (e) -- (f) -- (g);
    \end{tikzpicture}
    \caption{The graph $P_7$.}
    \label{P7}
\end{figure}

    Every fault tolerant zero forcing set contains a minimal zero forcing set, but not necessarily a minimum zero forcing set, as the next example demonstrates.

\begin{example}\label{ex:mummal}
    \begin{figure}[ht]
        \centering
        \begin{tikzpicture}[scale=2]
            \node[wzf] (1) at (0,0) {$1$};
            \node[wzf] (2) at (0,-1) {$5$};
            \node[wzf] (3) at (-.718,-.5) {$3$};
            \node[wzf] (4) at (1,-1) {$6$};
            \node[wzf] (5) at (1,0) {$2$};
            \node[wzf] (6) at (1.718,-0.5) {$4$};

            \draw[thick] (1) -- (3);
            \draw[thick] (2) -- (3);
            \draw[thick] (2) -- (4);
            \draw[thick] (4) -- (6);
            \draw[thick] (5) -- (6);
            \draw[thick] (1) -- (5);
            \draw[thick] (1) -- (2);
            \draw[thick] (4) -- (5);
        \end{tikzpicture}
        \caption{A connected graph on six vertices.}
        \label{fig:mummal}
    \end{figure}
    Consider the graph \Cref{fig:mummal}. The minimum zero forcing sets are $\{1,3\}, \ \{2,4 \}, \{3,5 \},$ and $\{4,6\}$ The set of vertices $B=\{1,2,5,6\}$ is a minimum fault tolerant set, but $B$ does not contain any of the minimum zero forcing sets.
\end{example}

The study of zero forcing and its variants has been augmented by the introduction of the notion of forts, which are obstructions to zero forcing sets. Specifically, a fort is a nonempty subset $F \subseteq V(G)$ such that for all $v \notin F, |N_G(v) \cap F| \neq 1$. Intuitively, a fort is a set of unfilled vertices such that when its complement is filled, no more forces are possible.

\begin{theorem}[\cite{brimkov2019forts},\cite{fast2018forts}]\label{thm:zforts}
    Let $G$ be a graph. A set of vertices $S \subseteq V(G)$ is a zero forcing set if and only if $S$ intersects every fort of $G$ nontrivially.
\end{theorem}

\begin{example}
Consider again the graph $G$ in \Cref{fig:mummal}.  There are seven minimal forts of $G$, namely $\{1, 4, 5\}$,
 $\{2, 3, 6\}$,
 $\{1, 2, 3, 4\}$,
 $\{1, 2, 5, 6\}$,
 $\{1, 3, 4, 6\}$,
 $\{2, 3, 4, 5\}$, and
 $\{3, 4, 5, 6\}$. One observes that each of the minimum zero forcing sets listed in \Cref{ex:mummal} intersects each fort.
\end{example}

\Cref{thm:zforts} extends naturally to the fault tolerant case.

\begin{proposition}
\label{obs:fortint}
    A set of vertices $S \subseteq V(G)$ is a $k$-fault tolerant zero forcing set if and only if $|S\cap F| \geq k+1$ for every fort $F$ of $G$.
\end{proposition}
\begin{proof}

    Let $S\subseteq V(G)$ be a set of vertices such that $|S|=m$, and denote by $S'$ an arbitrary subset of $S$ such that $|S'| = m-k$.
    First suppose $|S\cap F |\geq k+1$ for every fort $F$ of $G$. Then $|S' \cap F| \geq 1$ for every fort $F$ of $G$, so $S'$ is a zero forcing set by \Cref{thm:zforts}. Since the deletions were arbitrary, $S$ is a $k$-fault tolerant zero forcing set.
    
    Now suppose $S$ is a $k$-fault tolerant zero forcing set and $F$ is a fort of $G$. Then $|S' \cap F| \geq 1$ for every fort $F$ of $G$ by \Cref{thm:zforts}. Since otherwise one could delete all the vertices of $F$ from $S$ to contradict its fault tolerance, it must be the case that $|F|\geq k+1$ and $|S\cap F|\geq k+1$.
\end{proof}

For every graph $G$, the standard zero forcing number $\Z(G)$ exists.  However, when fault tolerance is enforced, the existence of $\Z_t^k(G)$ is not guaranteed.
\begin{corollary}
    Given a graph $G$, $\Z_t^k(G)$ exists if and only if every fort of $G$ has cardinality at least $k+1$.
\end{corollary}

When $k$ is small, it is possible to be more explicit about our characterization of the existence of $\Z_t^k(G)$. We first establish some simple characterizations of small forts.

\begin{proposition}\label{prop:smallforts}
    Let $G$ be a graph.  Then:
    \begin{enumerate}
        \item\label{prop:isolates} The only forts of $G$ of cardinality 1 are isolated vertices.
        \item\label{prop:twins} The forts of $G$ of cardinality 2 are precisely the sets of (adjacent or independent) twins.
    \end{enumerate}
        \end{proposition}
\begin{proof}
    \Cref{prop:isolates} follows from the fact if $\{v\}$ is a fort, then any neighbor $u$ of $v$ satisfies $|N_G(u)\cap \{v\}|=1$, and therefore $v$ has no neighbors. For \Cref{prop:twins}, it is clear that a pair of twins is a fort of cardinality 2. If $\{u,v\}$ is fort, then $N_G(u)=N_G(v)$ since any vertex $w$ in one neighborhood but not the other would satisfy $|N_G(w) \cap \{u,v\}|=1$.
\end{proof}

\begin{theorem}\label{thm:exist}
Let $G$ be a graph. Then:
\begin{enumerate}
    \item\label{obs:dne}
    $\Z_t(G)$ exists if and only if $G$ has no isolated vertices.
    \item\label{obs:dne2} $\Z_t^2(G)$ exists if and only if $G$ has neither isolated vertices nor twins.
    \item If $u,v$ are twins of $G$, then $\{u,v\}$ is a subset of every $k$-fault tolerant zero forcing set for $k\geq 1$.
\end{enumerate}
\end{theorem}

Note that in particular, \Cref{thm:exist} implies that if $k\geq 2$, then $\Z_t^k(G)$ does not exist if $G$ has a pair of twins.

\section{Graph Families}\label{sec:families}
We now establish values of the fault tolerant zero forcing number for graphs which belong to a common family of graphs.  The values of the zero forcing number for these graph families were established in \cite{AIM2008}. The path graph on $n$ vertices, denoted by $P_n$, has vertices $V(P_n)= \{ v_1,v_2,\dots, v_n\}$ and edges $E(P_n) = \{ v_iv_{i+1} : 1 \leq i \leq n-1 \}$.  The cycle graph $C_n$ on vertices is obtained from $P_n$ by adding the additional edge $v_1v_n$.  The wheel graph $W_n$ on $n$ vertices is obtained from $C_{n-1}$ by adding an additional vertex that is adjacent to all the vertices of the cycle. The complete bipartite graph $K_{m,n}$ is the graph with vertex set $V(K_{m,n}) = A\cup B$, where $A= \{ u_1,u_2, \dots, u_m\}$ and $B=\{v_1,v_2, \dots, v_n\}$, and edge set $E(K_{m,n})=\{ u_iv_j : 1\leq i\leq m, \ 1\leq j \leq n \}$. Examples are shown in \Cref{fig:graphs}. 

\begin{figure}[ht]
    \centering
    \begin{tikzpicture}[
    vertex/.style={circle, draw, fill=white, inner sep=1.5pt, minimum size=6pt},
    edge/.style={thin}
]

\begin{scope}[xshift=0cm]
    \foreach \i in {0,1} \node[bzf] (c\i) at ({60*\i}:1) {};
    \foreach \i in {2,3} \node[rzf] (c\i) at ({60*\i}:1) {};
    \foreach \i in {4,5} \node[wzf] (c\i) at ({60*\i}:1) {};
    
    \draw[edge] (c0)--(c1)--(c2)--(c3)--(c4)--(c5)--(c0);
    \node[below] at (0,-1.4) {$C_6$};
\end{scope}

\begin{scope}[xshift=4cm]
    \node[bzf] (w0) at (0,0) {};
    \foreach \i in {1,2} \node[bzf] (w\i) at ({60*\i}:1) {};
    \foreach \i in {3,4} \node[rzf] (w\i) at ({60*\i}:1) {};
    \foreach \i in {5,6} \node[wzf] (w\i) at ({60*\i}:1) {};
    \draw[edge] (w1)--(w2)--(w3)--(w4)--(w5)--(w6)--(w1);
    \foreach \i in {1,2,...,6} \draw[edge] (w0)--(w\i);
    \node[below] at (0,-1.4) {$W_7$};
\end{scope}

\begin{scope}[xshift=8cm]
    \node[wzf] (s0) at (0,0) {};
    \foreach \i in {1,2,...,5} \node[bzf] (s\i) at ({60*\i}:1) {};
    \node[rzf] (s6) at (0:1) {};
    \foreach \i in {1,2,...,6} \draw[edge] (s0)--(s\i);
    \node[below] at (0,-1.4) {$K_{1,6}$};
\end{scope}

\end{tikzpicture}
    \caption{Examples of a cycle, wheel, and complete bipartite graph. The blue vertices indicate a minimum zero forcing set, while the red vertices are the additional vertices required to produce a minimum $1$-fault tolerant zero forcing set.}
    \label{fig:graphs}
\end{figure}

\begin{proposition}\label{thm:pn}
Let $G = P_n$ be a path on $n\geq 2$ vertices.  Then $\Z_t(G) = 2$. If $n\geq 4$, then $\Z_t^2(G)=4$.
\end{proposition}

\begin{proof}
Let $G$ be a path on $n\geq2$ vertices.  Then $\Z(G)=1$ and by the \Cref{lowbound}  we have $\Z_t(G) \geq \Z(G) + 1  = 2$.  By choosing $B$ to consist of both endpoints, we are guaranteed that every subset of size 1 of $B$ is a zero forcing set, and thus $B$ is a fault tolerant zero forcing set.  Since this achieves the lower bound, we have $\Z_t(G) = 2$.

Now suppose $n\geq 4$.  The graph $G$ contains no 2-fault tolerant sets of size 3, since the only singleton zero forcing sets are the endpoints, and $G$ has only two endpoints.  If $B$ consists of a pair of adjacent degree 2 vertices and the degree 1 vertices, then $B$ is a 2-fault tolerant set since if $B\setminus \{v,w\}$ contains an endpoint it is a zero forcing set, but if $v,w$ are both endpoints the adjacent degree 2 vertices are a zero forcing set of $G$.
\end{proof}

\begin{proposition}\label{thm:kn}
If $G=K_n$ for $n\geq 2$, then $\Z_t(G) = n$.
\end{proposition}
\begin{proof}
    Let $G=K_n$ for some $n \geq 2$. Then $\Z(K_n) = n-1$ for all $n\geq 2$, and any set of $n-1$ vertices is a minimum zero forcing set.  Therefore, the set of all $n$ vertices is a fault tolerant zero forcing set, and thus $\Z_t(K_n) \leq n$.  We have $\Z_t(K_n) \geq n$ by \Cref{lowbound}.
\end{proof}

\begin{proposition}\label{thm:stars}
    If $G=K_{1,n}$ for $n\geq 3$, then $\Z_t(G) = n$.
\end{proposition}
\begin{proof}
    $\Z(G) = n-1$, and all minimum zero forcing sets are composed of all but one leaf.  Thus, the set of all leaves is a 1-fault tolerant set, so $\Z_t(G) \leq n$.  We have $\Z_t(G) \geq n$ by \Cref{lowbound}.
\end{proof}

\begin{proposition}\label{thm:Kmn}
    If $G=K_{m,n}$ for $n\geq m\geq 2$, then $\Z_t(G) = m+n$.
\end{proposition}
\begin{proof}
    Suppose $G=K_{m,n}$ for some $n\geq m\geq 2$. The set $B=V(G)$ is a fault tolerant zero forcing set of $G$ since only one force is required for each $B\setminus\{v\}$, so $\Z_t(G)\leq m+n$.

    Suppose there is a smaller fault tolerant zero forcing set $S$ of $G$. Denote the partite sets of $G$ by $A$ and $B$, then coloring the vertices of $S$ blue leaves at least one white vertex $v$ initially, assume without loss of generality $v\in A$.  There is always a vertex $w \in S$ such that there are two white vertices $v,p \in A$ after setting the vertices of $S\setminus \{w\}$ blue.  Since $N_G(v)=N_G(p)$, neither $v$ now $p$ can be forced and $S\setminus \{w\}$ is a failed zero forcing set.
\end{proof}

\begin{proposition}\label{thm:cycle}
    If $G=C_n$ for $n\geq 4$, then $\Z_t(G) = 4$. 
\end{proposition}
\begin{proof}
    The minimum (and minimal) zero forcing sets of $C_n$ consist of pairs of adjacent vertices, thus $\Z(C_n)=2$ for $n\geq 3$.  By \Cref{lowbound}, we have that $\Z_t(C_n) \geq 3$.  However, no set of three vertices guarantees that two adjacent vertices remain after one is removed.  An induced path of length 3 in $C_n$ preserves this property, and this is only possible if $n\geq 4$.
\end{proof}

\begin{proposition}\label{thm:wheels}
    If $G = W_n$ is a wheel on $n\geq 6$ vertices, then $\Z_t(G) = 5$.
\end{proposition}
\begin{proof}
   Let $G=W_n$ for $n \geq 6$. It was shown in \cite{Fallat2014MinimumRM} that $\Z(G)=3$. Note that the automorphism group of $G$ is the dihedral group acting on the long induced cycle. Therefore, up to isomorphism there are two minimum zero forcing sets, and these are precisely the minimal zero forcing sets. The first minimum zero forcing set is $B_1 = \{v_1,v_2,v_3\}$, where $\deg(v_1) = \deg(v_2) = 3$, $v_1v_2 \in E(G)$, and $\deg(v_3) = n-1$, while the second is $B_2 = \{w_1, w_2, w_3\}$, where $\deg(w_1) = \deg(w_2) = \deg(w_3) =3$ and $w_1w_2,w_2w_3 \in E(G)$. By \Cref{lowbound}, we have $\Z_t(G)\geq 4$. We will show there is no $B$ such that $|B| = 4$ is a fault tolerant zero forcing set and construct a fault tolerant zero forcing set $B$ such that $|B| = 5$.
    
   Suppose $B$ is a fault tolerant zero forcing set of $G$ and $|B|=4$.  Then $B$ must contain either $B_1$ or $B_2$ as a subset. Suppose $B = B_1 \cup \{ v_4\}$. If $v_4 \not\in N (\{{v_1, v_2}\})$, then $B\setminus \{v_1\} = \{v_2, v_3, v_4\}$ and each of these vertices has two white neighbors, and thus no forces are possible. If $v_4 \in N(\{{v_1, v_2}\})$, then assume without loss of generality that $v_1v_4 \in E(G)$. Again, every vertex of $B\setminus \{v_1\} = \{{v_2, v_3, v_4}\}$ has two white neighbors and no forces are possible. Therefore, there are no fault tolerant zero forcing sets of cardinality four that contain $B_1$.

    Now suppose $B = B_2 \cup \{ w_4\}$.  If $w_4$ is the universal vertex, we have already shown that $B$ is not fault tolerant.  If $w_iw_4 \not \in E(G)$ for $i=1,2,3$, then no forces are possible from $B\setminus \{w_1\}$, so without loss of generality assume $w_3w_4 \in E(G)$.  In this case $B\setminus \{w_2\}$ is not a zero forcing set, so $B$ cannot be fault tolerant.  Thus $\Z_t(G) \geq 5$.
    
    Now suppose $B = \{{v_1, v_2, v_3, v_4, v_5}\}$ where $G[\{{v_2, v_3, v_4, v_5}\}] \cong P_4$ and $v_1$ is a universal vertex. Consider $B\setminus \{{v_i}\}$ if $i = 2,3,4,5$ then $B\setminus \{v_i\}$ contains $B_1$ and is thus a zero forcing set. Similarly, $B\setminus \{v_1\}$ contains $B_2$, so it is a zero forcing set. Therefore  $B$ is a fault tolerant zero forcing set and $\Z_t(G) = 5$.
\end{proof}

\section{Compatible collections of path covers and applications to trees}\label{sec:pathcovers}

\subsection{Compatible path covers}

The path cover number of a graph provides an important upper bound on the zero forcing number of the graph, and in this section we derive the corresponding parameter for fault tolerant zero forcing.

\begin{definition}
    A path cover of a graph $G$ is a collection of induced paths $\{ P_i\}_{i=1}^r$ such that every vertex of $G$ lies in exactly one path $P_i$.
    The path cover number of a graph $G$, denoted $\mathrm{P}(G)$, is the minimum number of paths in a path cover of $G$.
\end{definition}

Given a zero forcing set $B$ of a graph $G$, during the forcing process it may be possible that more than one blue vertex is able to force the same white vertex.  Thus while the final set of blue vertices is unique for each initially blue set (whether it a zero forcing set or not), the forces used to obtain this final set of blue of vertices may not be. It is often convenient to work with a fixed set of forces, which leads to the next (well-known) definition.

\begin{definition}
    Let $G$ be a graph and $B$ a zero forcing set of $G$.  A sequence of forces $v_i \rightarrow v_{i+1}$ listed in the order in which they occur is called a chronological list of forces $\mathcal{F}$ for $B$ in $G$.  We omit  $B,G$ when they are clear from the context, and denote by $\mathcal{F} \setminus u\rightarrow v$ the list of forces obtained by deleting $u\rightarrow v$ from $\mathcal{F}$.
\end{definition}

Given a zero forcing set $B$ of a graph $G$, a chronological list of forces $\mathcal{F}$ for $B$ in $G$ specifies a path cover of $G$.  Each path in the path cover is of the form $G[\{v_1,v_2, \dots , v_r\}]$, where $v_i \rightarrow v_{i+1} \in \mathcal{F}$, $v_1\in B$ and $v_r$ does not perform a force in $\mathcal{F}$.  These are induced paths since each of the forces in $\mathcal{F}$ is valid. In particular, this is true for minimum zero forcing sets, which leads to the following lower bound on the zero forcing number of a graph.

\begin{theorem}[\cite{HOGBEN2010}, Theorem 2.13]\label{thm:pzg}
    Let $G$ be a graph.  Then $ \mathrm{P}(G) \leq \Z(G) $.
\end{theorem}

If $B$ is a fault tolerant zero forcing set, however, each $B\setminus \{v\}$ where $v \in B$ is a zero forcing set, and thus each vertex of $B$ induces a path cover of $G$. This suggests that collections of path covers are the natural object to consider in the fault tolerant case.  However, a number of subtleties must be addressed before we define the relevant parameter. Given a graph $G$ and a vertex $v\in V(G)$, we denote by $\mathrm{comp}_G(v)$ the connected component of $G$ which contains $v$. 

\begin{definition}
    Let $G$ be a graph and $\mathcal{P}=\{P_i\}_{i=1}^r$ a path cover of $G$. Define $L(\mathcal{P})$ to be the endpoints of paths in $\mathcal{P}$, that is $L(\mathcal{P}) =  \{ v \in V(G) \ | \ \deg_{\mathcal{P}}(v) \leq 1 \}$.  If $\mathfrak{P}$ is a set of path covers, then $L(\mathfrak{P})= \bigcup_{\mathcal{P}\in \mathfrak{P}} L(\mathcal{P})$.
\end{definition}

We next define the notion of a compatible collection of path covers. While the definition of a compatible collection is technical, the idea is that if a collection is compatible for $S$, then the vertices of $S$ can always be forced from two different path covers and is not required to initiate forcing on path. Thus $S$ is not required in a minimal fault tolerant zero forcing set.

\begin{definition}\label{def:compatiblecover}
    Let $G$ be a graph. A set $\mathfrak{P}=\{ \mathcal{P}_1, \mathcal{P}_2, \dots,\mathcal{P}_r\}$ of path covers of $G$ is called compatible for $S\subseteq L(\mathfrak{P})$ if for every $v\in S$, the following hold:
    \begin{enumerate}
        \item $\deg_{\mathcal{P}_i}(v) \geq 1$ for all $i \in \{1,2,\dots, r\}$, and
        \item for each $\mathcal{P}_i \in \mathfrak{P}$ such that $\deg_{\mathcal{P}_i}(v)=1$, let $u_i$ be the unique other endpoint of the path in $\mathcal{P}_i$ containing $v$. Then there exists some $\mathcal{P}_j \in \mathfrak{P}$ such that either:
        \begin{itemize}
            \item $\deg_{\mathcal{P}_j}(v) = 2$, or
            \item $\mathrm{comp}_{\mathcal{P}_j}(u_i) \neq \mathrm{comp}_{\mathcal{P}_j}(v)$, $\deg_{\mathcal{P}_j}(u_i) \geq 1$, and $S \cap L(\mathrm{comp}_{\mathcal{P}_j}(v)) = \{v\}$.
        \end{itemize}
    \end{enumerate}
\end{definition}
 We begin with a trivial example to show how compatible collections generalize path covers, before exhibiting a more interesting case.
\begin{example}
    Consider the graph $G$ in \Cref{P7}. Since $G$ is a path, trivially $\mathcal{P}=\{G\}$ is a path cover of $G$ and $\mathfrak{P}=\{\mathcal{P}\}$ is a compatible collection for each vertex in $\{v_2,v_3, \dots, v_5\}$ since $\deg(v_i)=2$ for $2\leq i \leq 5$.  However, $\deg(v_1) = \deg(v_7)=1$ and $v_1$ and $v_7$ are in the same component for every $\mathcal{P} \in \mathfrak{P}$, so $\mathfrak{P}$ is not compatible for $v_1$ and $v_7$.
\end{example}
\begin{figure}[ht]
    \centering
    \begin{tikzpicture}[scale=1.75,thick]
\tikzstyle{every node}=[minimum width=12pt, inner sep=2pt, circle]
    \draw (-0.59,2.35) node[draw] (0) {$8$};
    \draw (-0.59,1.69) node[draw] (1) {$1$};
    \draw (-1.29,1.69) node[draw] (2) {$2$};
    \draw (0.09,1.69) node[draw] (3) {$3$};
    \draw (-0.92,1.06) node[draw] (4) {$4$};
    \draw (-0.27,1.06) node[draw] (5) {$5$};
    \draw (-1.29,0.49) node[draw] (6) {$6$};
    \draw (0.09,0.49) node[draw] (7) {$7$};
    \draw  (0) edge (1);
    \draw  (1) edge (3);
    \draw  (1) edge (2);
    \draw  (1) edge (4);
    \draw  (1) edge (5);
    \draw  (3) edge (5);
    \draw  (4) edge (5);
    \draw  (2) edge (4);
    \draw  (4) edge (6);
    \draw  (5) edge (7);

\end{tikzpicture}
    \caption{The graph $E_1$ in \Cref{ex:E1}.}
    \label{fig:E1H7}
\end{figure}

\begin{example}\label{ex:E1}
    Consider the Graph $E_1$ in \Cref{fig:E1H7}. The collection $\mathfrak{P} = \{ \mathcal{P}_1,\mathcal{P}_2,\mathcal{P}_3,\mathcal{P}_4\}$, where each $\mathcal{P}_i$ is given in \Cref{fig:E1pcs}, is a set of path covers compatible for $\{1,3,4,5\}$.\begin{figure}[ht]
    \centering
        \begin{tikzpicture}[scale=1.25,thick]
\tikzstyle{every node}=[minimum width=10pt, inner sep=1pt, circle]
    \draw (-0.59,2.35) node[draw] (0) {$8$};
    \draw (-0.59,1.69) node[draw] (1) {$1$};
    \draw (-1.29,1.69) node[draw] (2) {$2$};
    \draw (0.09,1.69) node[draw] (3) {$3$};
    \draw (-0.92,1.06) node[draw] (4) {$4$};
    \draw (-0.27,1.06) node[draw] (5) {$5$};
    \draw (-1.29,0.49) node[draw] (6) {$6$};
    \draw (0.09,0.49) node[draw] (7) {$7$};
    \draw[thick,red!50] (0) --(1)--(2);
    \draw[thick,red!50] (6)--(4);
    \draw[thick,red!50] (7)--(5)--(3);
    \node (a) at (-.59,.2) {$\mathcal{P}_1$};

    \begin{scope}[shift={(3,0)}]
        \draw (-0.59,2.35) node[draw] (0) {$8$};
    \draw (-0.59,1.69) node[draw] (1) {$1$};
    \draw (-1.29,1.69) node[draw] (2) {$2$};
    \draw (0.09,1.69) node[draw] (3) {$3$};
    \draw (-0.92,1.06) node[draw] (4) {$4$};
    \draw (-0.27,1.06) node[draw] (5) {$5$};
    \draw (-1.29,0.49) node[draw] (6) {$6$};
    \draw (0.09,0.49) node[draw] (7) {$7$};
    \draw[thick,red!50] (0) --(1);
    \draw[thick,red!50] (6)--(4)--(2);
    \draw[thick,red!50] (7)--(5)--(3);
    \node (a) at (-.59,.2) {$\mathcal{P}_2$};
    \end{scope}
    \begin{scope}[shift={(6,0)}]
        \draw (-0.59,2.35) node[draw] (0) {$8$};
    \draw (-0.59,1.69) node[draw] (1) {$1$};
    \draw (-1.29,1.69) node[draw] (2) {$2$};
    \draw (0.09,1.69) node[draw] (3) {$3$};
    \draw (-0.92,1.06) node[draw] (4) {$4$};
    \draw (-0.27,1.06) node[draw] (5) {$5$};
    \draw (-1.29,0.49) node[draw] (6) {$6$};
    \draw (0.09,0.49) node[draw] (7) {$7$};
    \draw[thick,red!50] (0) --(1)--(3);
    \draw[thick,red!50] (6)--(4)--(2);
    \draw[thick,red!50] (7)--(5);
    \node (a) at (-.59,.2) {$\mathcal{P}_3$};
    \end{scope}
    \begin{scope}[shift={(9,0)}]
        \draw (-0.59,2.35) node[draw] (0) {$8$};
    \draw (-0.59,1.69) node[draw] (1) {$1$};
    \draw (-1.29,1.69) node[draw] (2) {$2$};
    \draw (0.09,1.69) node[draw] (3) {$3$};
    \draw (-0.92,1.06) node[draw] (4) {$4$};
    \draw (-0.27,1.06) node[draw] (5) {$5$};
    \draw (-1.29,0.49) node[draw] (6) {$6$};
    \draw (0.09,0.49) node[draw] (7) {$7$};
    \draw[thick,red!50] (0) --(1)--(3);
    \draw[thick,red!50] (7)--(5)--(4)--(6);
    \node (a) at (-.59,.2) {$\mathcal{P}_4$};
    \end{scope}

\end{tikzpicture}
\caption{A minimum compatible collection for $E_1$ in \Cref{ex:E1}.}
    \label{fig:E1pcs}
\end{figure}
\end{example}

\begin{definition}\label{def:essential}
    Let $\mathfrak{P}$ be a collection of path covers of $G$. Since $\mathfrak{P}$ is trivially compatible for $S=\emptyset$, the set
    \[
        \mathcal{S}(\mathfrak{P}) = \{\, S \subseteq L(\mathfrak{P}) \ : \ \mathfrak{P} \text{ is compatible for } S \,\}
    \]
    is always nonempty. Fix $S^* \in \mathcal{S}(\mathfrak{P})$ of maximum cardinality (such $S^*$ need not be unique, but $|S^*|$ is), and define the essential vertices of $\mathfrak{P}$ to be
    \[
        \Es(\mathfrak{P}) = L(\mathfrak{P}) \setminus S^*.
    \]

    A collection of path covers $\mathfrak{P}$ of a graph $G$ is a minimum compatible collection for $G$ if $|\Es(\mathfrak{P})| \leq |\Es(\mathfrak{Q})|$ for every collection of path covers $\mathfrak{Q}$ of $G$. The cardinality of a minimum compatible collection for $G$ is called the essential number of $G$, and is denoted by $\Es(G)$.
\end{definition}

Our next result shows that the essential number of a graph plays the same role in fault tolerant zero forcing that the path cover number plays in standard zero forcing.

\begin{theorem}\label{thm:IltZt}
    If $G$ is a graph for which $\Z_t(G)$ exists, then $\Es(G) \leq \Z_t(G)$.
\end{theorem}

\begin{proof}
    Let $B$ be a minimum fault-tolerant zero forcing set of $G$ with $|B| = \Z_t(G)$. 
    By fault tolerance, for every $v \in B$, the set $B \setminus \{v\}$ is a zero forcing set. 
    For each $v \in B$, fix a valid forcing sequence starting from $B \setminus \{v\}$ and let $\mathcal{P}_v$ be the path cover of $G$ induced by this sequence. 
    Define the collection $\mathfrak{P} = \{ \mathcal{P}_v : v \in B \}$.
    
    In our path cover representation, each path is initiated by coloring at least one of its endpoints. 
    Thus $B \setminus \{v\}$ contains at least one endpoint from each path in $\mathcal{P}_v$, implying $B \setminus \{v\} \subseteq L(\mathcal{P}_v)$. 
    Taking the union over $v \in B$ yields $L(\mathfrak{P}) \supseteq B$. 
    We call $S \subseteq L(\mathfrak{P})$ the set of redundant endpoints in $\mathfrak{P}$, i.e., $S = L(\mathfrak{P}) \setminus B$. 
    It suffices to verify that $\mathfrak{P}$ is compatible for $S$ according to \Cref{def:compatiblecover}. 
    If $S = \emptyset$, compatibility is vacuous, so assume $S \neq \emptyset$.

    \noindent\textit{Condition (1):} Let $u \in S$ and $\mathcal{P} \in \mathfrak{P}$. 
    Since $\mathcal{P}$ arises from a valid forcing sequence that colors all vertices of $G$, $u$ must appear in $\mathcal{P}$ either as an endpoint or an internal vertex. 
    Hence $\deg_{\mathcal{P}}(u) \geq 1$. This holds for all $\mathcal{P} \in \mathfrak{P}$, satisfying Condition (1).

    \noindent\textit{Condition (2):} Fix $u \in S$ and let $\mathcal{P}_i \in \mathfrak{P}$ be any cover with $\deg_{\mathcal{P}_i}(u) = 1$. 
    Let $u_i$ be the unique other endpoint of the path in $\mathcal{P}_i$ containing $u$. 
    Since $u \notin B$ and each path in $\mathcal{P}_i$ requires one initially colored endpoint to initiate forcing, we must have $u_i \in B \setminus \{v_i\} \subseteq B$. 
    Consequently, the backup cover $\mathcal{Q} := \mathcal{P}_{u_i}$ belongs to $\mathfrak{P}$. 
    We consider two cases:
    \begin{itemize}
        \item \textbf{Case 1:} $u$ is an internal vertex in $\mathcal{Q}$. Then $\deg_{\mathcal{Q}}(u) = 2$, satisfying alternative (2a) of \Cref{def:compatiblecover}.
        \item \textbf{Case 2:} $u$ remains an endpoint in $\mathcal{Q}$. Since $u \notin B$, it must be colored via the color-change rule in the forcing process for $\mathcal{Q}$. 
        In the path cover representation, this implies $u$ is forced by a sequence proceeding from an initially colored endpoint. 
        Because $u_i \notin B \setminus \{u_i\}$, the path containing $u$ in $\mathcal{Q}$ cannot be initiated from $u_i$. 
        Thus $u$ and $u_i$ lie on distinct paths in $\mathcal{Q}$, so $\mathrm{comp}_{\mathcal{Q}}(u_i) \neq \mathrm{comp}_{\mathcal{Q}}(u)$. 
        Moreover, the other endpoint of $\mathrm{comp}_{\mathcal{Q}}(u)$ must belong to $B$, yielding $S \cap L(\mathrm{comp}_{\mathcal{Q}}(u)) = \{u\}$. 
        Taking $\mathcal{P}_j = \mathcal{Q}$ satisfies alternative (2b) of \Cref{def:compatiblecover}.
    \end{itemize}
    In both cases, a valid backup cover $\mathcal{P}_j \in \mathfrak{P}$ exists. 
    Hence $\mathfrak{P}$ is compatible for $S$, and by definition $\Es(G) \leq |L(\mathfrak{P}) \setminus S| = |B| = \Z_t(G)$.
\end{proof}

Returning to \Cref{ex:E1}, we see that $\Es(\mathfrak{P}) = \{1,2,7,8\}$. This is indeed a fault tolerant zero forcing set, and it is minimal by \Cref{thm:IltZt} (or \Cref{lowbound} and the fact that $\Z(E_1)=3$).

In general, determining whether a set of compatible path covers for a graph $G$ is minimum is difficult and is a topic that requires further study. All path covers can be enumerated and thus an algorithm exists to test minimality, but this is not efficient to implement.  Some simplifications exist, however, and as an example we offer the next result.

\begin{lemma}
    If $\mathfrak{P}$ is a minimum compatible collection for $G$ and $v \in L(\mathfrak{P})$ is isolated in every path cover of $\mathfrak{P}$, then $\mathfrak{P}-v$ is a minimum compatible collection of $G-v$. In particular, under these hypotheses, $\Es(G) = \Es(G-v)+1$
\end{lemma}
\begin{proof}
    Suppose to the contrary that $\mathfrak{Q}$ is a compatible collection for $G-v$ and $\Es(\mathfrak{Q}) < \Es(\mathfrak{P})-1$. Since $v$ was essential in $\mathfrak{P}$, adding $v$ to each path cover of $\mathfrak{Q}$ yields a compatible collection for $G$ with cardinality less than $\Es(\mathfrak{P})$, a contradiction.
\end{proof}

\subsection{Trees}
Trees are connected acyclic graphs, and we previously dealt with the special cases of $P_n$ and $K_{1,n}$. In this section we give an easily computable bound on the fault tolerant zero forcing number of a tree $T$, then show that the sets of compatible path covers of defined earlier in this section determine $\Z_t(T)$.

When we restrict our attention to trees, the inequality of \Cref{thm:IltZt} becomes equality, and establishing this is the primary goal of this section. We first recall the corresponding result in the standard zero forcing case.

\begin{theorem}[\cite{param}]\label{thm:pz}
    Let $T$ be a tree.  Then $\mathrm{P}(T) = \Z(T)$.
\end{theorem}

In fact, there is a correspondence between (minimum) zero forcing sets of a tree $T$ and (minimum) path covers of $T$, obtained by taking one endpoint of each path in a path cover.  This yields a zero forcing set of $T$, and all zero forcing sets are obtained in this way (see, for example, the proof of Proposition 4.2 in \cite{AIM2008}).

We denote by $\ell(T)$ the number of vertices of degree 1 in a tree $T$, and refer to $\ell(T)$ as the leaf number of $T$. For leaky forcing, it was shown in \cite[Proposition 3.2]{alameda2024leaky} that the set of leaves of a tree with an edge is the unique 1-leaky forcing set of $T$. This implies the next result, but we provide our own proof for self containment.

\begin{proposition}\label{thm:leaves}
    The set of leaves of a tree $T \not\cong K_1$ is a fault tolerant set of $T$.  Hence, \\ $\Z_t(T) \leq \ell(T)$.
\end{proposition}
\begin{proof}
We know from \cite{geneson2020reconfiguration} that every tree has a zero forcing set that consists entirely of leaves.  We must show that any set of all but one leaf is a zero forcing set of $T$. 

For trees, a zero forcing set is obtained by choosing one endpoint of every path in a path cover of $T$. Let $B=\{ v \in V(T) \ \vert \ v \text{ is a leaf of } T \}$.  We consider $B\setminus \{v\}$ for some $v \in B$ and construct a path cover such that each path contains an endpoint in $B\setminus \{v\}= \{ w_1, w_2 , \dots, w_r\}$. Note that $r\geq 1$ since $T\not\cong K_1$.  Furthermore, since $T$ is a tree there is unique path between any two vertices of $T$.  Let ${P}_1$ be the unique path between $w_1$ and $v$. For $j \in \{2,3,\dots,r\}$, let $P_j$ be the maximum induced subpath of the unique path from $w_j$ to $v$ that $V(P_j)\cap (\cup_{i=1}^{j-1} V(P_i)) = \emptyset$. Clearly each path in $\mathcal{P} = \{P_1, P_2, \dots, P_r\}$ contains an endpoint.  To see that $\mathcal{P}$ is a path cover of $T$, suppose $u \in V(T)$ s.t. $u \not\in V(P_i)$ for every $i \in \{1,2,\dots, r\}$. Then $u$ cannot be a leaf, so $\deg_T(u)\geq 2$. Consider $x,y \in N_T(u)$ such that $x\neq y$, then there exists a unique path $Q_1$ to $v$ and a unique path $Q_2$ from $y$ to $v$.  Then $T[V(Q_1)\cup V(Q_2)\cup\{u\}]$ is a cycle which is a contradiction. Therefore $\mathcal{P}$ is a path cover of $T$ where every path contains at least one vertex of $B\setminus \{v\}$ as an endpoint. Thus $B\setminus \{v\}$ is a zero forcing set.

\end{proof}

The star $K_{1,n}$ and the path $P_n$ show that the bound is tight.  Inequality can be achieved, however, as the next example shows.
\begin{example}\label{ineq}
    Let $G$ be the graph depicted in \Cref{counter}.  There are five leaves and thus $\ell(G)=5$, but $\Z_t(G) \leq 4$, as evidenced by the initially blue set $B = \{ 1,2,6,7 \} $. 
\end{example}

\begin{figure}[ht]
    \centering
    \begin{tikzpicture}
        \node[wzf] (1) at (0,0) {$1$};
        \node[wzf] (2) at (0,-2) {$2$};
        \node[wzf] (3) at (1,-1) {$3$};
        \node[wzf] (4) at (3,-1) {$4$};
        \node[wzf] (5) at (5,-1) {$5$};
        \node[wzf] (6) at (6,0) {$6$};
        \node[wzf] (7) at (6,-2) {$7$};
        \node[wzf] (8) at (3,0) {$8$};

        \draw[thick] (1) -- (3)-- (4) -- (5) -- (6);
        \draw[thick] (2) -- (3);
        \draw[thick] (7) -- (5);
        \draw[thick] (8) -- (4);
        
    \end{tikzpicture}
    
    \caption{The tree $T$ in \Cref{ineq} and \Cref{ex:mftpc}.}
    \label{counter}
\end{figure}

\begin{theorem}
\label{thm:endpoints}
    Let $T$ be a tree with at least two vertices. Then $\Z_t(T) = |\Es(T)|$.
\end{theorem}
\begin{proof}
    In light of \Cref{thm:IltZt}, we need only show that $\Es(T) \geq \Z_t(T)$.
    Thus, let $T$ be a tree and $\mathfrak{P} = \{\mathcal{P}_i\}_{i=1}^r$ a minimum compatible collection for $T$. Let $B=\Es(\mathfrak{P})$. We claim that $B$ is a fault tolerant zero forcing set of $T$, and proceed by induction on the order of $T$. When the order of $T$ is 2, we have $T = K_2$. There are two path covers of $K_2$ with the same endpoints, and thus we have $\Es(\mathfrak{P})=V(K_2)$ which is clearly a fault tolerant set. Now assume the result holds for trees of order up to $|T|-1$.
    
    For some $v \in B$, we consider $B \setminus \{v\}$. As stated previously, it is well-known that taking one endpoint of any path cover of $T$ yields a zero forcing set of $T$. Therefore if $\deg_{\mathcal{P}_i}(v) =1$ for some $i \in \{1,2,\dots, r\}$, then $B\setminus \{v\}$ contains an endpoint of every path in $\mathcal{P}_i$ and is thus a zero forcing set of $T$.  Now suppose $v$ is isolated in every path cover $\mathcal{P}_i \in \mathfrak{P}$. 


    Let $w$ be a neighbor of $v$ in $T$.  If $w\in B$ (and therefore $B\setminus \{v\}$) then for each $\mathcal{P}_i$, $w$ is either isolated or contained in a component with both endpoints. If $w$ is isolated in some $\mathcal{P}_i$, then since $B \setminus\{v\}$ is a zero forcing set of $T-v$, $w$ is able to force $v$ after the color change rule is applied to $T-v$, so $B\setminus \{v\}$ is a zero forcing set of $T$.  Now suppose $w$ is never isolated in $\mathfrak{P}$. Then $\deg_{\mathcal{P}_i}(w) \geq 1$ for all $i$. If $\deg_{\mathcal{P}_i}(w)=2$ for all $i$ then $w \not \in \Es(\mathfrak{P})$ as hypothesized, thus there exists some $i$ such that $\deg_{\mathcal{P}_i}(w) = 1$, and the other endpoint of $\mathrm{comp}_{\mathcal{P}_i}(w)$ is in $B\setminus \{v\}$.  Then as before, $w$ can force $v$ after all forces are carried out in $T-v$.

    If $ w \not \in B$, then $\mathfrak{P}$ is compatible for $w$. If there exists some $i$ such that $\deg_{\mathcal{P}_i}(w) = 1$, and the other endpoint of $\mathrm{comp}_{\mathcal{P}_i}(w)$ is in $B\setminus \{v\}$.  So suppose $\deg_{\mathcal{P}_i}(w)=2$ for all $i$. Then both endpoints of $\mathrm{comp}_{\mathcal{P}_i}(w)$ are in $B\setminus \{v\}$, so we can choose a forcing starting from one endpoint of $\mathrm{comp}_{\mathcal{P}_i}(w)$ and terminating at $w$, while the other endpoint initiates a forcing process coloring the remainder of $\mathrm{comp}_{\mathcal{P}_i}(w)$. Then as before, $w$ can force $v$ after all forces are carried out in $T-v$.

    Thus we have shown that $B$ is a fault tolerant zero forcing set, and therefore $\Z_t(T) \leq \Es(T)$.

\end{proof}

\begin{example}\label{ex:mftpc}
    Let $T$ be the tree in \Cref{counter}. Then $\mathfrak{P}= \{\mathcal{P}_1,\mathcal{P}_2\}$, where $\mathcal{P}_1 = \{ T[\{1,3,4,8\}] , T[\{5,6,7\}], T[\{2\}] \}$ and $\mathcal{P}_2 = \{ T[\{ 1,2,3\}], T[\{4,5,6,8\}], T[\{7\}]\}$ (shown in \Cref{fig:mftpc}) is a set of compatible path covers for the vertices $\{3,4,5,8\}$ in $T$. Furthermore, $\mathfrak{P}$ is a minimum compatible collection for $T$, and thus $\Es(\mathfrak{P}) = \{1,2,6,7\}$ is a minimum fault tolerant set of $T$ and $\Z_t(T)=4$.
\end{example}

Note that the choice of path covers is not arbitrary.  For instance, let $\mathcal{P}_3 = \{ T[\{ 1,2,3\}], T[\{4,8\}], T[\{5,6,7\}]\}$. Then neither $\{\mathcal{P}_1,\mathcal{P}_3\}$ nor $\{\mathcal{P}_2,\mathcal{P}_3\}$ is compatible for a set of four vertices.

\begin{figure}[ht]
    \centering
    \begin{tikzpicture}[scale=0.75]
        \node[wzf] (1) at (0,0) {$1$};
        \node[wzf] (2) at (0,-2) {$2$};
        \node[wzf] (3) at (1,-1) {$3$};
        \node[wzf] (4) at (3,-1) {$4$};
        \node[wzf] (5) at (5,-1) {$5$};
        \node[wzf] (6) at (6,0) {$6$};
        \node[wzf] (7) at (6,-2) {$7$};
        \node[wzf] (8) at (3,0) {$8$};

        \draw[thick,red!50] (1) -- (3)-- (2);
        \draw[thick,red!50] (8) -- (4)--(5)--(6);

        \begin{scope}[shift={(10,0)}]
        \node[wzf] (1) at (0,0) {$1$};
        \node[wzf] (2) at (0,-2) {$2$};
        \node[wzf] (3) at (1,-1) {$3$};
        \node[wzf] (4) at (3,-1) {$4$};
        \node[wzf] (5) at (5,-1) {$5$};
        \node[wzf] (6) at (6,0) {$6$};
        \node[wzf] (7) at (6,-2) {$7$};
        \node[wzf] (8) at (3,0) {$8$};

        \draw[thick,red!50] (7) -- (5) -- (6);
        \draw[thick,red!50] (1) -- (3) -- (4) -- (8);
        \end{scope}
        
    \end{tikzpicture}
    
    \caption{Two path covers of the tree $T$ in \Cref{ineq}.}
    \label{fig:mftpc}
\end{figure}

We next demonstrate that the difference between the leaf number and the fault zero forcing number can be made arbitrarily large for trees. Let $T_1$ denote the graph in \Cref{counter}. Note that all vertices of $T_1$ are of degree 3 or degree 1, and the degree 3 vertices induce a path. We obtain $T_k$ from $T_{k-1}$ by subdividing an edge between two degree 3 vertices, and attaching a leaf to the newly created vertex. This process is repeated a second time, so a total of 4 new vertices are added, two degree 3 and two of degree 1. We display $T_2$ and $T_3$ in \Cref{fig:tk}. We call $T_k$ in this family the $k$-th special caterpillar. We note that $|V(T_k)| = 4k+4$, and that $\ell(T_k) = 2k+3$.
\begin{figure}[ht]
    \centering
    \begin{tikzpicture}[scale=0.75]

        \node[wzf] (1) at (0,0) {};
        \node[wzf] (2) at (0,-2) {};
        \node[wzf] (3) at (1,-1) {};
        \node[wzf] (9) at (2,-1) {};
        \node[wzf] (10) at (4,-1) {};

        \node[wzf] (4) at (3,-1) {};
        \node[wzf] (5) at (5,-1) {};
        \node[wzf] (6) at (6,0) {};
        \node[wzf] (7) at (6,-2) {};
        \node[wzf] (8) at (3,0) {};
        \node[wzf] (11) at (2,0) {};
        \node[wzf] (12) at (4,0) {};

        \draw[thick] (1) -- (3)--(9)--(4)--(10) -- (5) -- (6);
        \draw[thick] (2) -- (3);
        \draw[thick] (7) -- (5);
        \draw[thick] (8) -- (4);
        \draw[thick] (9) -- (11);
        \draw[thick] (10) -- (12);
        \node at (3,-1.75) {$T_2$};
        \begin{scope}[xshift = 10cm]
            
        \node[wzf] (1) at (-2,0) {};
        \node[wzf] (2) at (-2,-2) {};
        \node[wzf] (3) at (-1,-1) {};
        \node[wzf] (9) at (2,-1) {};
        \node[wzf] (10) at (4,-1) {};

        \node[wzf] (4) at (3,-1) {};
        \node[wzf] (5) at (5,-1) {};
        \node[wzf] (6) at (6,0) {};
        \node[wzf] (7) at (6,-2) {};
        \node[wzf] (8) at (3,0) {};
        \node[wzf] (11) at (2,0) {};
        \node[wzf] (12) at (4,0) {};
        \node[wzf] (13) at (1,0) {};
        \node[wzf] (14) at (0,0) {};
        \node[wzf] (15) at (1,-1) {};
        \node[wzf] (16) at (0,-1) {};

        \draw[thick] (1) -- (3)--(16)--(15)--(9)--(4)--(10) -- (5) -- (6);
        \draw[thick] (2) -- (3);
        \draw[thick] (7) -- (5);
        \draw[thick] (8) -- (4);
        \draw[thick] (9) -- (11);
        \draw[thick] (10) -- (12);
        \draw[thick] (13) -- (15);
        \draw[thick] (14) -- (16);
        \node at (2,-1.75) {$T_3$};
        \end{scope}
        
    \end{tikzpicture}    \caption{The graphs $T_2$ and $T_3$.}
    \label{fig:tk}
\end{figure}

We will make use of the following result, which allows for the computation of the path cover number of a tree. If $H$ is a subgraph of $G$, denote by $G \setminus V (H)$ is the induced subgraph of $G$ obtained by deleting the vertices of $H$.
\begin{theorem}[\cite{kim2009smith}, Proposition 13]\label{thm:mpcover}
    Let $T$ be a tree that is not a path. Suppose that $P$ is a path in $T$ such
that $\deg_P(v)=1$ implies $\deg_T(v)=1$. Moreover, suppose $P$ has exactly one vertex $v$ of
degree 3 or more in $T$ . Then
\[\mathrm{P}(T ) = \mathrm{P}(T \setminus V (P)) + 1.\] 
\end{theorem}

\begin{proposition}\label{prop:TkMPC}
    Let $T_k$ be the $k$-th special caterpillar. Then $\mathrm{P}(T_k) = k+2$.
\end{proposition}
\begin{proof}
    We proceed by induction on $k$.
    For $k=1,2$, it is straightforward to verify that $\mathrm{P}(T_1)=3$ and $\mathrm{P}(T_2)=4$. Suppose $\mathrm{P}(T_{k-1}) = (k-1)+2$ for some $k\geq 2$; we show $\mathrm{P}(T_{k+1}) = (k+1)+2$.  Consider $v \in V(T_{k+1})$, where $\deg_{T_{k+1}}(v) =3$ and $N_{T_{k+1}}(v)$ contains independent twins. Note that there are exactly two choices for $v$ in $T_{k+1}$, denote them by $v_1$ and $v_2$ and the twins in their neighborhood by $u_1, w_1$ and $u_2,w_2$, respectively. Let $P_i = T[\{u_i,v_i,w_i\}]$ for $i=1,2$, and consider the graph $S=(T\setminus P_1)\setminus P_2$. We claim that $\mathrm{P}(S) = \mathrm{P}(T_{k-1})$.  Indeed, $S$ has exactly two vertices of degree 2 (i.e. the remaining neighbors of $v_1$ and $v_2$, and by Lemma 2.3 in \cite{sher2025enumeration}, these vertices are always in the same path as their degree 1 neighbors, call them $l_1$ and $l_2$ respectively. Thus contracting along the edges $v_1l_1$ and $v_2l_2$ does not change the path cover number, but these contractions yield a graph isomorphic to $T_{k-1}$. By the induction hypothesis, $k+1 = \mathrm{P}(T_{k-1}) = \mathrm{P}(S)$, and two applications of \Cref{thm:mpcover} yields $\mathrm{P}(T_{k+1})=k+3$.
\end{proof}
\begin{proposition}\label{prop:cat}
    Let $T_k$ be the $k$-th special caterpillar. Then:
    \begin{enumerate}
        \item\label{it:Z} $\Z(T_k) = k+2$, and
        \item\label{it:Zt} $\Z_t(T_k) = k+3$.
    \end{enumerate}
\end{proposition}
\begin{proof}
    \Cref{it:Z} follows from \Cref{thm:pz} and \Cref{prop:TkMPC}, so we prove \Cref{it:Zt}.  By \Cref{lowbound}, $\Z_t(T_k) \geq k+3$, so exhibiting a fault tolerant forcing set of cardinality $k+3$ suffices to prove the result. Let $P$ denote the induced path of degree 3 vertices, and let $\{ v_1, v_2, \dots ,v_{2k-1}\}$ be the set of vertices with $\deg_P(v_i)=2$. Label the degree 1 neighbor in $T_k$ of $v_i$ by $l_i$, the two unnamed neighbors of $v_1$, $v_{2k-1}$ by $w_1$ and $w_2$, respectively,and the remaining leaves $r_j$ for $1\leq j \leq 4$ with $r_1,r_2\in N_{T_k}(w_1)$ and $r_3,r_4\in N_{T_k}(w_2)$, and (see \Cref{fig:Tklabel}).  We claim that there is a set of path covers compatible for $S=\{v_1,v_3,\dots, v_{2k-1}\}$.
    Define 
    \begin{align*}
    \mathcal{P}_1 &= \left\{T_k[\{r_2\}],T_k[\{ r_1, w_1 , v_1, l_1\}] ,T_k[\{ l_2, v_2 , v_3, l_3\}],\dots, T_k[\{ l_{2k-2}, v_{2k-2} , v_{2k-1}, l_{2k-1}\}], T_k[\{ r_3, w_2 , r_4\}]\right\}, \\
    \mathcal{P}_2 &= \left\{T_k[\{ r_1, w_1 , r_2\}] ,T_k[\{ l_1, v_1 , v_2, l_2\}],\dots,T_k[\{ l_{2k-3}, v_{2k-3} , v_{2k-2}, l_{2k-2}\}], T_k[\{ l_{2k-1}, v_{2k-1}, w_2, r_3 \}], T_k[\{ r_4\}]\right\}.
    \end{align*}
    For $\mathfrak{P}=\{\mathcal{P}_1,\mathcal{P}_2\}$, we have that $L(\mathfrak{P})$ is the set of all leaves of $T_k$. Let $S=\{v_1,v_3,\dots, v_{2k-1}\}$. We claim $\mathfrak{P}$ is compatible for $S$.  Clearly each $v\in S$ satisfies $\deg_{\mathcal{P}_i}(v) = 1$ for $i=1,2$, as the only singleton paths in $\mathfrak{P}$ are those induced by $r_2$ and $r_4$, and each vertex of $S$ is a leaf. Moreover, each vertex of $S$ lies in a path in $\mathcal{P}_1$ that is distinct from the path in $\mathcal{P}_2$, and the other endpoint of the path is not in $S$.  Thus $\mathfrak{P}$ is compatible for $S$. That $S$ is a minimum compatible path cover follows from \Cref{lowbound}, and $\Z_t(T_k)= |L(\mathfrak{P})\setminus S| = k+3$ by \Cref{thm:endpoints}.
\end{proof}

\begin{figure}[ht]
    \centering
    \begin{tikzpicture}[scale=1.5]
        \node[wzf] (1) at (-2,0) {$r_1$};
        \node[wzf] (2) at (-2,-2) {$r_2$};
        \node[wzf] (3) at (-1,-1) {$w_1$};
        \node[wzf] (10) at (4,-1) { $v_{2k-1}$};
        \node[wzf] (4) at (3,-1) {$v_{2k-2}$};
        \node[wzf] (5) at (5,-1) {$w_2$};
        \node[wzf] (6) at (6,0) {$r_3$};
        \node[wzf] (7) at (6,-2) {$r_4$};
        \node[wzf] (8) at (3,0) {$l_{2k-2}$};
        \node[wzf] (12) at (4,0) {$l_{2k-1}$};
        \node[wzf] (13) at (1,0) {$l_2$};
        \node[wzf] (14) at (0,0) {$l_1$};
        \node[wzf] (15) at (1,-1) {$v_2$};
        \node[wzf] (16) at (0,-1) {$v_1$};
        \draw[thick] (1) -- (3)--(16)--(15);
        \draw[thick] (4)--(10) -- (5) -- (6);
        \draw[thick] (2) -- (3);
        \draw[thick] (7) -- (5);
        \draw[thick] (8) -- (4);
        \draw[thick] (10) -- (12);
        \draw[thick] (13) -- (15);
        \draw[thick] (14) -- (16);
        \draw[thick,loosely dashed] (15)--(4);
        \node at (2,-1.75) {$T_k$};
    \end{tikzpicture}
    \caption{Illustration of the vertex naming convention in the proof of \Cref{prop:cat}.}
    \label{fig:Tklabel}
\end{figure}
\begin{corollary}
    Let $T_k$ be the $k$-th special caterpillar.  Then $\ell(T_k)- \Z_t(T_k) = k$.
\end{corollary}

\section{Graph Operations} \label{sec:ops}

We now consider the effect of deleting a vertex or an edge from a graph on the fault tolerant zero forcing number.  Recall for a graph $G$ and $e$ an edge of $G$, we denote by $G-e$ the graph obtained by deleting $e$ from $G$ and by $G\slash e$ the graph obtained by contracting along $e$.  For a vertex $v$ of $G$, we denote by $G-v$ the graph obtained by deleting $v$ and any incident edges to $v$ from $G$. We use software to verify several of the calculations in this section. An implementation to compute values of fault tolerant zero forcing numbers in Sage is available at \cite{ftcode}.

For zero forcing, the effects of these operations were studied in \cite{edholm2012} and \cite{owens}.  We begin by recalling the results therein.

\begin{theorem}[\cite{edholm2012},\cite{owens}]
    Let $G$ be a graph, $e$ an edge of $G$, and $v$ a vertex of $G$.  
    \begin{enumerate}
        \item $-1 \leq \Z(G) -\Z(G-e) \leq 1$
        \item $-1 \leq \Z(G) - \Z(G\slash e) \leq 1$
        \item $-1 \leq \Z(G) -\Z(G-v) \leq 1$
    \end{enumerate}
\end{theorem}

The proof of these results relies on \Cref{zfb}, which is no longer true for fault tolerant zero forcing.
\begin{theorem}[\cite{param}]\label{zfb}
Let $G$ be a graph.  If $G$ is connected and has at least two vertices, then no vertex is in every minimal zero forcing set.
\end{theorem}

For example, we will show that all twins are required in every fault tolerant zero forcing set in \Cref{ztgminusv}.  Thus, looser bounds hold in the fault tolerant case.

\begin{theorem}\label{thm:ftgme}
    Let $G$ be a graph such that $\Z_t(G)$ exists and $e$ an edge of $G$ whose deletion leaves no isolated vertices. Then
    
        \[-2 \leq \Z_t(G) -\Z_t(G-e) \leq 2.\]

\end{theorem}
\begin{proof}
    Suppose $G$ and $e$  satisfy the hypotheses, let $u$ and $v$ be the endpoints of $e$, let $B$ be a fault tolerant zero forcing set of $G - e$ and $\mathcal{F}_w$ be a chronological list of forces for $B\setminus \{w\}$ in $G-e$. 

    Let $B^{\prime} = B \cup \{u,v\}$. For each $\mathcal{F}_w$, construct $\mathcal{F}_w^{\prime}$ by removing any force received by $u$ or $v$. Each force in $\mathcal{F}_w^{\prime}$ is valid in $G$ since at each time step the number of white neighbors of each vertex of $G$ cannot be higher than the number of white neighbors at that time step in $G-e$.  Thus $B^{\prime}$ is a fault tolerant zero forcing set for $G$, and $|B^{\prime}| \leq |B|+2$. Thus $\Z_t(G) \leq \Z_t(G-e) +2$.

    Now suppose $B$ is a fault tolerant zero forcing set of $G$ with chronological list of forces $\mathcal{F}_w$ for $B\setminus \{w\}$. Construct $B^{\prime} = B \cup \{u,v\}$ and $\mathcal{F}_w^{\prime}$ from $\mathcal{F}_w$ by deleting any force received by $u$ or $v$. As before, each force in $\mathcal{F}_w^{\prime}$ is valid in $G-e$, so $B^{\prime}$ is a fault tolerant zero forcing set of $G-e$. Since $|B^{\prime}| \leq |B|+2$, we have $\Z_t(G-e) \leq \Z_t(G)+2$.

\end{proof}

\begin{figure}[ht]
    \centering
    \begin{tikzpicture}
        \node[wzf] (1) at (0,1.5) {$1$};
        \node[wzf] (2) at (1.5,1.5) {$2$};
        \node[wzf] (3) at (3,1.5) {$3$};
        \node[wzf] (4) at (4.5,1.5) {$4$};
        \node[wzf] (5) at (0,3) {$5$};
        \node[wzf] (6) at (1.5,3) {$6$};
        \node[wzf] (7) at (3,3) {$7$};
        \node[wzf] (8) at (4.5,3) {$8$};
        \node[wzf] (9) at (3,4.5) {$9$};

        \draw[thick] (1) -- (2) node [below,midway] {$a$};
        \draw[thick] (1) -- (6) node [near start,left] {$b$};;
        \draw[thick] (5) -- (2);
        \draw[thick] (5) -- (6);
        \draw[thick] (7) -- (6);
        \draw[thick] (2) -- (7);
        \draw[thick] (2) -- (3) node [below,midway] {$e$};;
        \draw[thick] (7) -- (8);
        \draw[thick] (3) -- (4) node [below,midway] {$d$};
        \draw[thick] (7) -- (3) node [left,midway] {$c$};;
        \draw[thick] (4) -- (7);
        \draw[thick] (7) -- (9);
    \end{tikzpicture}
    \caption{The graph in \Cref{ex:gme}.}
    \label{fig:gme}
\end{figure}

\begin{example}\label{ex:gme}
   To show that all of the possibilities of the value of $\Z_t(G)$ and $\Z_t(G-e)$ implied by \Cref{thm:ftgme} can be realized, consider the graph $G$ in \Cref{fig:gme}. Note that no graph of order less than 9 exhibits all possibilities. Computation of $\Z_t(G)=4$ and $\Z_t(G-e)$ for various edges $e$ establishes the following: \\

   \begin{table}[ht]
       \centering

   \begin{tabular}{c|c}
       $\mathbf{e}$ & $\mathbf{\Z_t(G)-\Z_t(G-e)}$ \\ \hline
         a & 2 \\
         b& 1 \\
         c & 0 \\
         d & -1 \\
         e & -2
   \end{tabular}\\
   \end{table}

\end{example}

    \begin{theorem}\label{ztgslashe}
        Let $G$ be a graph  such that $\Z_t(G)$ exists,  and $e$ an edge of $G$ such that $\Z_t(G\slash e)$ exists. Then

            \[-2 \leq \Z_t(G) -\Z_t(G\slash e) \leq 2.\]
       
    \end{theorem}

\begin{proof}
    Suppose $u$ and $v$ are the endpoints of $e$, denote the quotient map from $G$ to $G\slash e$ by $\varphi$, and let $u^{\prime}=\varphi(u) = \varphi(v) $. Let $B$ be a minimum fault-tolerant zero-forcing set of $G\slash e$. Suppose first $u^{\prime} \not\in B$. Let $\mathcal{F}$ be a chronological list of forces for $B\setminus\{{w}\}$ in $G\slash e$ for some $w \in B$. Then each force in $\mathcal{F}$ is valid until $x \rightarrow u^{\prime} $ for some $x \in V (G\slash e)$. Let $F^{\prime}$ be the chronological list of forces that is identical to $\mathcal{F}$ until this force. After applying the forces in $F^{\prime}$, $x$ has at most has two white neighbors in $G$, namely $u$ and $v$. 
 
    Note that since $u^{\prime} \not\in B$, $\varphi$ is bijective on $B$.  Let $D = \varphi^{-1} (B) \cup \{{u,v}\}$. Consider $D \setminus \{{w}\}$ for some $w \in D$. After performing all the forces in $F^{\prime}$, $x$ has at most one white neighbor. If $x$ has exactly one white neighbor, without loss of generality $u$, add $x \rightarrow u$ to $F^{\prime}$. 
    
    If $u^{\prime} \rightarrow y \in \mathcal{F}$, for some $y \in V (G\slash e)$, then $y$ is the only white vertex in $N(u) \cup N(v)$. So $u \rightarrow y$ or $v \rightarrow y$ is valid. Add whichever force is valid to $\mathcal F^{\prime}$, and the rest of the forces after $u^{\prime} \rightarrow y$ in $\mathcal F$ are valid in $G$, so we append them to $\mathcal F^{\prime}$. Thus $D \setminus\{{w}\}$ is a zero forcing set for $G$ with chronological list of force $\mathcal F^{\prime}$. Since $w$ was arbitrary, $D$ is a fault tolerant zero forcing set of $G$.

    Now consider the case where $u ^\prime \in B$. Let $\mathcal F$ be a chronological list of forces for $B \setminus \{w\}$ in $G \slash e$ where $w \neq u ^\prime$. If $u^\prime \rightarrow c \in \mathcal F$ and occurs at position $j$ in $\mathcal F$, then define $B^\prime = \varphi ^ {-1} (B \setminus \{u^\prime\}) \cup \{u,v,c\}$, or $B^\prime = \varphi ^ {-1} (B \setminus \{u^\prime\}) \cup \{u,v\}$ if $u^\prime$ does not perform a force in $\mathcal F$. 
    
    Consider $B^\prime \setminus \{r\}$. If $\{r\} \cap \{u,v\} = \emptyset$, then since $N_{G\slash e} (u^\prime) = N_G(u) \cup N_G(v)$ replace $u^\prime \rightarrow c$ if it exists with $u \rightarrow c$ or $v \rightarrow c$ at position $j$ (at least one of which must be valid). 
    
    Consider $B^ \prime \setminus \{r\}$. If $\{r\} \cap \{u,v\} \neq \emptyset$, then without loss of generality suppose $\{r\} \cap \{u,v\} = \{u\}$. $N(u) \cup N(v)$ has exactly one white vertex, namely $u$. Replace $u^\prime \rightarrow c$, if it exists, with $v \rightarrow u$ at position $j$. Since $|B^\prime| \leq |B| +2$, then $\Z_t (G) \leq \Z_t (G \slash e) +2$.

    For the other inequality, let $B$ be a minimum fault tolerant zero forcing set for $G$. We will construct a fault tolerant zero forcing set $ B^{\prime} \supset\varphi(B)$. Let $\mathcal F$ be a chronological list of forces for $B \setminus \{{w}\}$ in $G$, where $w\in B$. Each force $a \rightarrow b \in \mathcal F$ is valid in $G \slash e$ as long as $\{{a,b}\} \cap \{{u,v}\} = \emptyset$. Let $\mathcal F^{\prime}$ be the chronological list of forces that is identical to $\mathcal F$ except for the deletion of forces of this form.
    
    There are several cases to consider. Suppose first that $a \rightarrow u\in \mathcal F$, where  $a \neq v $. Since $a$ has one white neighbor in $G$, $a$ has at most one white neighbor in $G \slash e$, namely $u^{\prime}$. In this case, add $a \rightarrow u^{\prime}$ to $\mathcal F^{\prime}$ in the location where $a\rightarrow u$ was in $\mathcal F$. 
    
    Now suppose $u \rightarrow v \in \mathcal F$, then after the preceding forces in $\mathcal F$ occur, $u$ is blue. Then at this point in $\mathcal F^{\prime}$, $u^{\prime}$ is also blue. Suppose $u \rightarrow c, v \rightarrow d \in \mathcal F$. Let $B^{\prime} = \varphi (B) \cup \{{c,d}\}$ so that the remaining forces for $B \setminus \{{w}\}$ are all valid for $B^{\prime}$ in $G\slash e$ where $  w \not\in \{c, d\}$. Consider $B^{\prime} \setminus \{{c}\}$ (the case of $B^{\prime}\setminus \{d\}$ is similar). 
    
    
    Every force occurring before $u \rightarrow c$ in $\mathcal F$ is valid in $G \slash e$.  Furthermore, since $d$ is filled, $u^{\prime}$ has at most one white neighbor after these forces have been performed. Add $u^{\prime} \rightarrow c$  to $F^{\prime}$ in the location where $u \rightarrow c$ occurred in $\mathcal F$.  Then $\mathcal F^{\prime}$ is a chronological list of forces of $B^{\prime} \setminus \{c\}$ in $G \slash e$, so $B^{\prime} \setminus \{c\}$ is a zero forcing of $G \slash e$.  This resolves the remaining cases to show that $B^{\prime}$ is a fault tolerant zero forcing set of $G\slash e$, and thus $\Z_t (G\slash e) \leq \Z_t (G) +2$.
    \end{proof}

    \begin{figure}[ht]
        \centering
        \begin{tikzpicture}[scale=2]
            \node[wzf] (1) at (0,0) {$1$};
            \node[wzf] (2) at (1,0) {$2$};
            \node[wzf] (3) at (-.5,-1) {$3$};
            \node[wzf] (4) at (1.5,-1) {$4$};
            \node[wzf] (5) at (.5,-1.5) {$5$};
            \node[wzf] (6) at (-1,-2) {$6$};
            \node[wzf] (7) at (1,-2) {$7$};
            \node[wzf] (8) at (2,-2) {$8$};
            \node at (.5,-2.5) {(a)};

            \draw[thick] (1) -- (2);
            \draw[thick] (4) -- (2);
            \draw[thick] (1) -- (3);
            \draw[thick] (3) -- (4) node [below,midway] {$a$};
            \draw[thick] (1) -- (5) node [left,midway] {$c$};
            \draw[thick] (5) -- (2);
            \draw[thick] (3) -- (6) node [left,midway] {$b$};
            \draw[thick] (3) -- (5);
            \draw[thick] (5) -- (7);
            \draw[thick] (7) -- (4) node [left,midway] {$d$};
            \draw[thick] (6) -- (5);
            \draw[thick] (4) -- (8);
            \draw[thick] (5) -- (8);

            \begin{scope}[shift={(3,-.5)}]
                \node[wzf] (1) at (0,0) {$1$};
                \node[wzf] (2) at (1,0) {$2$};
                \node[wzf] (3) at (0,-1) {$3$};
                \node[wzf] (4) at (1,-1) {$4$};
                \node at (.5,-2) {(b)};
                \draw[thick] (1) -- (2) -- (4) -- (3) -- (1);
                \draw[thick] (2) -- (3) node [left,midway] {$f$};

            \end{scope}
            
        \end{tikzpicture}        
        \caption{The graphs in \Cref{ztgcontre}.}
        \label{fig:ztgcontre}
    \end{figure}
    
\begin{example}\label{ztgcontre}
    Consider the graph $G$ in \Cref{fig:ztgcontre} (a). Computation of $\Z_t(G)$ and $\Z_t(G\slash e)$ for various edges $e$ establishes the following: \\

   \begin{table}[ht]
       \centering

   \begin{tabular}{c|c}
       $e$ & $\mathbf{\Z_t(G)-\Z_t(G\slash e)}$ \\ \hline
         $a$ & -2 \\
         $b$& -1 \\
         $c$ & 0 \\
         $d$ & 1
   \end{tabular}\\
   \end{table}

   Meanwhile, for the graph $H$ in \Cref{fig:ztgcontre} (b), we have $\Z_t(H) - \Z_t(H \slash f) =2$. 
\end{example}

    \begin{theorem}\label{ztgminusv}
        Let $G$ be a graph such that $\Z_t(G)$ exists and $v$ a vertex of $G$ whose deletion leaves no isolated vertices. Then

            \[-\deg(v)\leq \Z_t(G) -\Z_t(G-v) \leq 2.\]
       
    \end{theorem}
\begin{proof}
    Let $G$ and $v$ satisfy the hypotheses and let $B$ be a fault tolerant zero forcing set of $G-v$ with corresponding chronological list of forces $\mathcal{F}_w$ for each $B\setminus\{w\}$, where $w\in B$.  Let $v_1 \in B$. Then there is a first force $x\rightarrow a$ in $\mathcal{F}_{v_1}$, where $x \in N_G(v)$.  Construct $B^{\prime} = B \cup \{v,a\}$ and for each $u \in B$ we let $\mathcal{F}_u^{\prime} = \mathcal{F}_u \setminus y\rightarrow a$, where $y$ is the vertex of $G-v$ that forces $a$.  Each of these forces is valid and results in all vertices blue when starting with $B^{\prime}\setminus \{u\}$ initially blue. 

    It remains to show that $B^{\prime}\setminus \{v\}$ is a zero forcing set. Before any forces are performed, $v$ is the unique white neighbor of $x$, and therefore $x\rightarrow v$ is a valid first force.  We then proceed with the forces in $\mathcal{F}^{\prime}_{v_1}$ to color the remaining vertices.  Thus $B^{\prime}$ is a fault tolerant zero forcing set with $|B^{\prime}| \leq |B|+2$, so $\Z_t(G) \leq \Z_t(G-v) +2$.

    For the other inequality, let $B$ be a minimum fault tolerant zero forcing set of $G$.  Since $v$ can only force its neighbors, it is immediate that $B^{\prime}=B \cup N_G(v)$ is a zero forcing set of $G-v$. Thus we need to show fault tolerance. As before, for each $w\in B$ let $\mathcal{F}_w$ be a chronological list of forces for $B\setminus \{w\}$ in $G$. Let $\mathcal{F}^{\prime}_w$ be the list of forces obtained by deleting forces of the form $x \to u$ for each $u\in N_G(v)$ such that $u \neq w$. Then each force in $\mathcal{F}^{\prime}_w$ is valid for $B^{\prime} \setminus \{w\}$, so $B^{\prime}$ is a fault tolerant zero forcing set.
\end{proof}
\begin{figure}[!ht]
    \centering
    \begin{tikzpicture}
        \node [wzf] (1) at (0, 0) { $1$};

		\node [wzf] (2) at (1, -1) { $2$};

		\node [wzf] (3) at (1, 1) { $3$};

		\node [wzf] (4) at (2, 0) {$4$ };

		\node [wzf] (5) at (3.5, 0) {$5$ };
		\draw[thick] (4)--(3)--(1)--(2)--(4)--(5);
    \end{tikzpicture}
    \caption{The graph in \Cref{ex:gmv}.}
    \label{fig:gmv}
\end{figure}

\begin{theorem}\label{zttwinv}
    Let $G$ be a graph with twins $u$ and $v$.  Then $0 \leq \Z_t(G) - \Z_t(G-v) \leq 2$.
\end{theorem}
\begin{proof}

    We have $\Z_t(G) - \Z_t(G-v) \leq 2$ by \Cref{ztgminusv}, thus it remains to show that $\Z_t(G-v) \leq \Z_t(G)$.  Let $B$ be a minimum fault tolerant zero forcing set of $G$, and $B$ contains both $u$ and $v$. Consider $B\setminus \{x\}$ for some $x \in B$ and a chronological list of forces $\mathcal F$ for $B \setminus \{v,x\}$. We create $\mathcal F^{\prime}$ from $\mathcal F$ by replacing a force of the form $v\rightarrow w$ with $u \rightarrow w$ if it appears, and if $x=v$, removing the force where $v$ is forced.   Then $\mathcal F^{\prime}$ is a chronological list of forces for $B\setminus \{x,v\}$ in $G-v$ that forces all vertices.  Thus $B \setminus \{v\} $ is a fault tolerant zero forcing set.
\end{proof}

Note that by applying \Cref{thm:exist}, $ \Z_t(G-v)$ exists only if deleting $v$ does not result in any isolated vertices.

\begin{example}\label{ex:gmv}
        Consider the graph $G$ in \Cref{fig:gmv}. Computation of $\Z_t(G)$ and $\Z_t(G-v)$ for various vertices $v$ establishes the following: \\
   \begin{table}[ht]
       \centering

   \begin{tabular}{c|c}
       $\mathbf{v}$ & $\mathbf{\Z_t(G)-\Z_t(G-v)}$ \\ \hline
         5 & 0 \\
         1& 1 \\
         2 & 2 
   \end{tabular}\\
   \end{table}
\end{example}

The lower bound on $\Z_t(G) - \Z_t(G-v)$ cannot be improved, as the next example shows.

\begin{example}
    Let $G$ be a graph with $V(G) = \{ v_1, v_2 , \dots, v_{2n+1}\}$ and $E(G) = \{ v_1v_i \ | \ i =2,3,\dots, n+1\} \cup \{ v_iv_{n+i} \ | \ 2\leq i \leq n+1\}$.  This graph is commonly called a generalized star.  Since $G$ is a tree with path cover number $n-1$, we have $\Z(G) = n-1$. It is straightforward to verify that $\Z_t(G) = n$ since the degree 1 vertices form a fault tolerant set.  However, $\Z_t(G-v_1) = 2n$, since $G-v_1$ consists of the disjoint union of $n$ paths.
\end{example}

We note that \Cref{zttwinv} could be used to shorten the proofs of \Cref{thm:kn} - \Cref{thm:Kmn}. In the case of graphs with several twins, it also reduces the search space in brute-force calculations of the fault tolerant zero forcing number.

\section*{Acknowledgments}
We thank the referees for their close reading on the article which has greatly improved the exposition.

\bibliographystyle{siam}
\bibliography{ZFT2}
\newpage
\section*{Author Information}
Asher Brown \\
Winston-Salem State University \\
506 Fieldstone Rd \\
Mooresville, NC 28115 \\
 abrown423@rams.wssu.edu  

 \medskip

\noindent Mark Hunnell\\
Winston-Salem State University \\
2621 Audubon Dr \\
Winston-Salem, NC 27106 \\
hunnellm@wssu.edu 

\medskip

\noindent Za'Kiya Toomer-Sanders \\
Winston-Salem State University \\
2414 Quantum Ct \\
Winston-Salem, NC 27106 \\
ztoomersanders121@rams.wssu.edu

\medskip

\noindent Via Weber \\
Winston-Salem State University \\
530 N Patterson Ave, Unit 151 \\
Winston-Salem, NC 27101 \\
sweber124@rams.wssu.edu

\section*{Financial Declarations}

The work of the authors was partially supported by National Science Foundation Award no.2447261.
\end{document}